\documentclass[preprint,10pt]{elsarticle}

\usepackage[utf8]{inputenc}
\usepackage{setspace}
\usepackage{lmodern}
\usepackage{amsmath,amssymb}
\usepackage{amsthm}
\theoremstyle{plain}
\usepackage{caption}
\usepackage{hyperref}
\usepackage{cleveref}
\usepackage{enumitem}

\usepackage[a4paper,margin=1in]{geometry}
\newtheorem{assumption}{Assumption}[section]

\crefname{assumption}{assumption}{assumptions}
\Crefname{assumption}{Assumption}{Assumptions}

\newtheorem{proposition}{Proposition}[section] 

\newtheorem{theorem}{Theorem}[section]

\newtheorem{remark}{Remark}[section]

\usepackage{graphicx}
\usepackage{subfigure}
\usepackage{epsf,epsfig}
\usepackage{bm,bbm}
\usepackage{diagbox}
\usepackage{algpseudocode}
\usepackage{algorithmicx, algorithm}
\usepackage{physics}
\usepackage{color}
\usepackage{pxfonts}
\usepackage[font=footnotesize,labelfont=bf]{caption}
\usepackage{mathtools}
\usepackage{hhline}
\usepackage[T1]{fontenc}
\usepackage{placeins}

\usepackage{multirow,dcolumn}
\usepackage{mathrsfs}
\usepackage{verbatim}
\usepackage{comment}

\usepackage{epstopdf}
\usepackage{lineno}

\renewcommand{\arraystretch}{1.5}
\graphicspath{{figures/}}
\journal{Elsevier}

\begin{document}

\begin{frontmatter}
    \title{A persistent-homology-Gaussian prior for solving infinite-dimensional Bayesian inverse scattering problems
    }
    \author[TJU]{Zhiyuan Wang}
     \author[TJU]{Hang Qi}
    \author[TJU]{Xiaofei Guan\corref{cor}}
    \ead{guanxf@tongji.edu.cn}
    \author[UESTC]{Zhiliang Deng\corref{cor}}
    \ead{dengzhl@uestc.edu.cn}
    \author[SJU]{Xiaomei Yang}
    \address[TJU]{School of Mathematical Sciences, Tongji University, Shanghai 200092, China}
    \address[UESTC]{School of Mathematical Sciences, University of Electronic and Science Technology of China, Chengdu 610054, China}
     \address[SJU]{School of Mathematics, Southwest Jiaotong University, Chengdu 610097, China}
    \cortext[cor]{Corresponding Author}    
    
    \begin{abstract}
 Bayesian inference methods have been developed to address inverse problems in function spaces where the unknown parameters are of infinite dimension. However, conventional Gaussian priors remain inadequate for reconstructing discontinuous or sharply varying target functions encountered in practical applications like obstacle reconstruction. Although hybrid priors have emerged as a promising solution, significant challenges remain in developing theoretically rigorous and computationally tractable frameworks in engineering applications. To address these issues, we propose a persistent-homology-Gaussian (PHG) prior for solving the acoustic obstacle scattering inverse problem in the infinite-dimensional Bayesian setting, which combines a weighted persistence-based regularization term with a periodic Gaussian reference measure through a Gibbs tilt. Then, the complex boundary is represented by a log-radial function on the unit circle, so that the reconstruction from far-field data is formulated as a function-space inverse problem. The well-posedness of the resulting posterior measure is established in the Hellinger, total variation, and Wasserstein-\(p\) metrics. Furthermore, the convergence of finite-dimensional posterior approximations is obtained, and posterior sampling is performed by a preconditioned Crank--Nicolson (pCN) method. Numerical experiments show that the proposed PHG prior yields accurate and stable reconstructions under more extensive noisy conditions, providing explicit control of multiscale topological features and better performance compared to other conventional priors.

\end{abstract}

\end{frontmatter}


\section{Introduction}
Inverse obstacle scattering seeks to recover an unknown scatterer from measurements of wave fields generated by given incident waves. Such problems arise in many important applications, including radar and sonar imaging \cite{borden2002mathematical}, medical ultrasound tomography \cite{kuchment2013radon}, and nondestructive testing \cite{bao2002inverse}. From a mathematical perspective, inverse obstacle scattering is typically nonlinear and ill-posed in the sense of Hadamard: small perturbations in the data may lead to large errors in the reconstructed scatterer unless suitable regularization is imposed \cite{colton1998inverse}. Moreover, since the unknown boundary is naturally described by a function rather than by a finite set of parameters, the reconstruction problem is infinite-dimensional. Classical numerical methods for inverse obstacle scattering include iterative reconstruction schemes such as Newton-type methods \cite{Monch1996} and Landweber iteration \cite{Hettlich1998}, as well as multifrequency continuation and recursive linearization methods for quantitative recovery \cite{Bao2015,doi:10.1137/16M1093562}. Non-iterative approaches include sampling-based methods such as the linear sampling method \cite{Colton1996}, the factorization method \cite{Kirsch1998}, the direct sampling method \cite{Li2013}, and extended sampling strategies for limited-aperture data \cite{Liu2018}. Related inverse scattering methods based on finite-element and reciprocity-gap \cite{monk2007inverse} techniques, as well as multiple-frequency eigenvalue approaches \cite{sun2012eigenvalue}, have also been developed. See also the survey articles \cite{KRESS1995171} for broader reviews. These approaches have achieved considerable success in shape reconstruction. However, they are primarily designed to produce point estimates and generally do not provide the quantification of uncertainty.

Bayesian inversion has emerged as a powerful framework for inverse problems because it offers a unified probabilistic formulation that incorporates observational noise and prior information and yields quantitative uncertainty estimates \cite{stuart2010inverse,kaipio2005statistical}. For the problems of function-valued unknowns, this is particularly natural, as the posterior can be defined directly on an infinite-dimensional parameter space rather than only after discretization \cite{stuart2010inverse,dashti2013bayesian}. Foundational results on posterior well-posedness and stability were established in \cite{stuart2010inverse}, while more recent formulations under weaker assumptions were developed in \cite{latz2020well,latz2023bayesian} (see \cite{sullivan2015introduction,calvetti2018inverse} for more details). For acoustic obstacle reconstruction, Bui-Thanh \cite{bui2014analysis} established an infinite-dimensional Bayesian framework  and derived corresponding finite-dimensional approximations.  
Subsequent work considered more challenging measurement settings, including phaseless far-field data generated by point source waves \cite{Yang_2020,YANG2021114073} and limited-aperture observations \cite{yin2025physics,Yang13102023} . Alternative geometric descriptions, such as point-cloud-based representations of the scatterer, were also explored within Bayesian approaches \cite{palafox2017point}, and related Bayesian formulations were further developed for interior cavity scattering problems \cite{wang2015bayesian}. Furthermore, some computational acceleration strategies for Bayesian inverse problems have also been given, such as adaptive multi-fidelity polynomial chaos surrogates and local-approximation-based Stein variational methods \cite{yan2019adaptive,yan2021stein}. These studies show that Bayesian modeling provides a flexible framework for inverse scattering. They also indicate that once the boundary is treated as a function-valued unknown, reconstruction quality depends crucially on how prior information is encoded.

In function-space Bayesian inversion, the prior plays a central role: it not only incorporates available knowledge about the unknown parameter, but also regularizes the inverse problem and shapes the resulting posterior distribution \cite{stuart2010inverse,bui2014analysis}. Gaussian priors are widely used because they fit naturally into the infinite-dimensional Bayesian framework and are compatible with function-space MCMC methods such as pCN \cite{cotter2013mcmc,cui2016dimension}. However, they tend to over-smooth solutions, making them inadequate for capturing sharp transitions or highly oscillatory boundaries. To better capture limited smoothness and edge-preserving features, a variety of non-Gaussian or hybrid constructions have been introduced, including Besov priors and TV (total variation)-typed Gaussian hybrid priors \cite{dashti2011besov,yao2016tv,lv2020nonlocal}. For inverse problems with unknown interfaces or shapes, Bayesian level-set methods \cite{iglesias2016bayesian} have provided a flexible framework for geometric reconstruction. In acoustic scattering, Bayesian shape reconstruction has been studied on shape spaces with Gaussian priors \cite{bui2014analysis}, and through angular or radial boundary parameterizations equipped with different priors \cite{yang2019bayesian,kuijpers2024wavenumber}. In addition, finite-dimensional prior constructions guided by topological sensitivity analysis have been explored for inverse scattering \cite{carpio2020bayesian}. These priors are mainly designed to encode smoothness, edge information, or coarse geometric features, but they do not explicitly characterize multiscale topological structure. Persistent homology provides a natural framework for characterizing such structure, it can describe the evolution of connected components, the presence of closed loops, and the appearance and disappearance of cavities across different scales \cite{edelsbrunner2008persistent,chazal2021introduction}. Recent work \cite{deng2025bayesian} has shown that persistence-based descriptors can be incorporated into prior construction through hybrid formulations relative to a Gaussian reference measure. Motivated by this idea, we introduce a persistent-homology-Gaussian (PHG) prior for acoustic obstacle reconstruction.

This paper studies two-dimensional acoustic obstacle reconstruction from far-field data. We introduce a PHG prior that captures the topology of the obstacle's radial boundary while retaining a Gaussian reference measure. The proposed PHG prior is different from the approach in \cite{carpio2020bayesian}, which constructs the topological prior using the topological derivative of the cost function. In contrast, the method introduced in this paper captures the topological structure in the unknown scatterer, and can be viewed as a topological generalization of total-variation (TV) priors. The main contributions of the paper can be summarized as follows:

\begin{itemize}

\item The hybrid Persistent-homology-Gaussian (PHG) prior \cite{deng2025bayesian} is developed for inverse scattering problems by incorporating the persistent homology analysis under the Bayesian framework. Moreover, the obstacle geometry is characterized by the periodic radial functions to reduce the dimension of computation, which provides better robustness and flexibility through additional tunable parameters.

\item The theoretical analysis of the hybrid PHG prior is also established, and the well-posedness of the resulting posterior measure in the infinite dimensional setting is derived under several statistical metrics, including the Hellinger, total variation, and Wasserstein-$p$ metrics.

\item  Theoretically, the function space is extended to a generalized topological space, and its finite-dimensional realization is also established. Furthermore, the PHG prior is gradient-free, which preserves the more complex geometric information through the persistent homology of filtrations induced by the shape representation, and TV prior preserves local geometric features through variation-based regularization in $BV$-type spaces.

\item   Several numerical examples using a pCN scheme are presented to demonstrate the effectiveness of the proposed PHG prior. The results show that the PHG prior can improve the reconstruction quality compared to other conventional methods under more extensive noisy conditions.
\end{itemize}

The rest of this paper is organized as follows. In 
~\Cref{se2}, Bayesian inverse scattering problems are given in detail. Then, the construction of PHG prior is obtained in ~\Cref{sec:topo-prior}. In ~\Cref{sec:wellposed}, we derived
some theoretical results of the hybrid PHG prior, including the well-posedness and several important properties. In ~\Cref{se5}, some numerical experiments are presented to demonstrate the effectiveness of the proposed method. Finally, conclusions are given in ~\Cref{se6}.

\section{Bayesian inverse scattering problems}\label{se2}
In this section, we formulate the two-dimensional acoustic inverse obstacle scattering problem and the associated Bayesian framework. We also introduce a periodic Gaussian reference prior and establish the well-posedness of the corresponding posterior measure. This provides the reference measure and the analytical foundation for the PHG prior studied in the next section.

\subsection{Problem setup}
\par 
Let $\Omega \subset \mathbb{R}^2$ be a bounded simply connected sound-soft obstacle with boundary
$
\Gamma := \partial \Omega
$
of class $C^{2,\alpha}$ for some $\alpha \in (0,1]$, and let
$
\Omega^{\mathrm e} := \mathbb{R}^2 \setminus \overline{\Omega}
$
denote the exterior scattering domain.  For an incident plane wave $
u^i(x,d)=e^{i\kappa x\cdot d}$ with propagation direction $ d\in S^1$ and wavenumber $\kappa=\omega/c>0$, the scattered field $u^s=u^s(x,d)$ satisfies the following equations

\begin{subequations}
\begin{align}
&\Delta u^s + \kappa^2 u^s = 0, \quad \text{in} \; \Omega^{\mathrm e},\tag{2.1a}\label{eq:2.1a}\\[4pt]
&u = 0,\quad \text{on} \; \Gamma,\tag{2.1b}\label{eq:2.1b}\\[4pt]
&\lim_{r \to \infty} \sqrt{r}\!\left(\frac{\partial u^s}{\partial r} - i \kappa u^s\right) = 0, && \tag{2.1c}\label{eq:2.1c}
\end{align}
\end{subequations}
where ~\cref{eq:2.1c} denotes the Sommerfeld radiation condition, ensuring outgoing wave behavior at infinity. Eq. (2.1) admits a unique solution, with the scattered field \(u^s\) exhibiting the asymptotic expansion \cite{colton1998inverse}
\begin{equation}\label{eq:2.2}
u^s(x,d) \;=\; 
\frac{e^{i\pi/4}}{\sqrt{8\kappa\pi}}\,
\frac{e^{i\kappa r}}{\sqrt{r}}
\left\{ u^\infty(\hat{x},d) + \mathcal{O}\!\left(\tfrac{1}{r}\right) \right\},
\qquad r := |x| \to \infty.
    \tag{2.2}
\end{equation}
This holds uniformly for all directions $\hat{x} = x/|x|$. The function $u^\infty(\hat{x},d)$ appearing in the expansion is known as the far-field pattern, which can be computed by solving the exterior Dirichlet problem for \(\phi\) on \(\Gamma\) under the following equation
\begin{equation}\label{eq:2-3}
u^\infty(\hat{x}, d) = \frac{e^{-i\pi/4}}{\sqrt{8\pi \kappa}} \int_{\Gamma} (\kappa \nu(y) \cdot \hat{x} + \zeta) e^{-ik\hat{x} \cdot y} \phi(y) ds(y),
\tag{2.3}
\end{equation}
where $\nu$ is the outward unit normal.

Given far-field data \(u^\infty(\hat{x},d)\) for \((\hat{x},d) \in \gamma^{o} \times \gamma^{i} \subset S^{1} \times S^{1}\), where \( \gamma^o \), \( \gamma^i \) denote the sets of observation and incident directions. The inverse scattering problem is to 
recover the shape of the obstacle
$\Omega$. In this work, we focus on the fixed-incidence, full-aperture setting, namely \(\gamma^{i} = \{d_{0}\}\) for some direction \(d_{0}\) and \(\gamma^{o} = S^{1}\). A schematic illustration of the obstacle, the incident aperture, and the observation aperture is shown in Figure~\ref{fig:geometry}.

\begin{figure}[htp]
    \centering
    \includegraphics[width=0.3\linewidth]{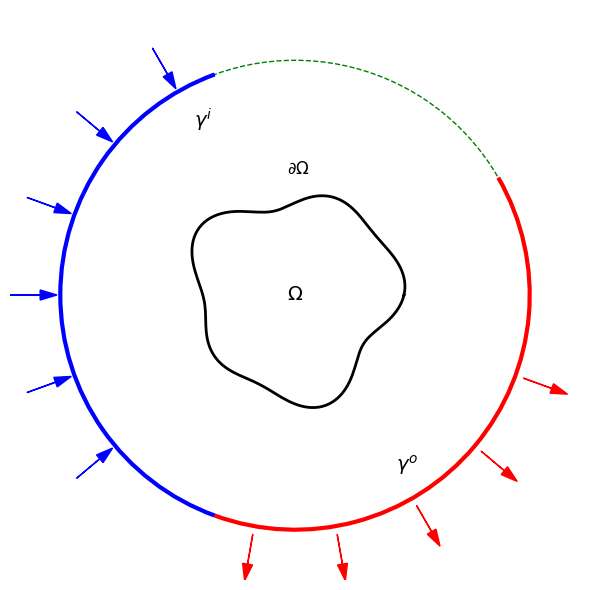}
    \caption{Obstacle $\Omega$, incident aperture $\gamma_i$, observation aperture $\gamma_o$.}
    \label{fig:geometry}
\end{figure}

To parametrize the unknown obstacle, we assume that $\Omega$ is star-shaped with respect to a given reference point $z_0\in\mathbb{R}^2$. Let
 $
e_r(\theta):=(\cos\theta,\sin\theta)^T $, $\theta\in[0,2\pi].
 $
We represent the boundary by
 $$
\Gamma_q := \left\{ z_0 + e^{q(\theta)} e_r(\theta) : \theta\in[0,2\pi] \right\},
\label{eq:boundary-param}
 $$
where $q(\theta)$ is the logarithmic radial function, we will refer to $q(\theta)$ as the log-radial function in the following sections. The corresponding obstacle is
 $$
\Omega_q := \left\{ z_0 + \rho e_r(\theta) : 0\le \rho < e^{q(\theta)},\ \theta\in[0,2\pi] \right\}.
\label{eq:obstacle-param} $$
Let
$
X:=C_{\mathrm{per}}^{2,\alpha}(S^1),
$
which can be defined as
\[
X:=C_{\mathrm{per}}^{2,\alpha}(S^1)=\left\{
q\in C^{2,\alpha}([0,2\pi]) : q^{(j)}(0)=q^{(j)}(2\pi),\ j=0,1,2
\right\}.
\]
Throughout this work, the unknown log-radial function is assumed to belong to $X$. Here, $\Gamma_q$ is a closed $C^{2,\alpha}$ boundary.

\begin{remark}
In classical numerical and theoretical analyses, it is required that the unknown boundary $\Gamma_q$ possesses at least two continuous derivatives with Hölder continuity of order \( \alpha \in (0,1] \) \cite{bui2014analysis,colton1998inverse}. Boundaries with regularity weaker than $C^2$ can still be addressed within the Riesz--Fredholm framework, we omit unnecessary technicalities  here.
\end{remark}

\subsection{Bayesian inversion framework}
In this subsection, the Bayesian inversion framework will be built. Fix the incident direction $d_0\in S^1$ and choose observation directions $\{\hat{x}_k\}_{k=1}^K\subset S^1$. For each $q(\theta)\in X$, let
$
\mathcal{F}(q):=u^\infty(\,\cdot\,,d_0;\Omega_q)
$
denote the corresponding far-field pattern. Here, we assume that the measured data are the magnitudes of the far-field pattern, and therefore define the finite-dimensional observation operator$$
\mathcal{G}(q)
:=
\bigl(
|\mathcal{F}(q)(\hat{x}_1)|,\dots,|\mathcal{F}(q)(\hat{x}_K)|
\bigr)^T
\in \mathbb{R}^K.
\label{eq:observation-operator}
$$
The measurement model is
\begin{equation}
y=\mathcal{G}(q^\dagger)+\zeta.\tag{2.4}\label{eq:data-model}
\end{equation}
where $q^\dagger\in X$ denotes the truth and
$
\zeta\sim \mathcal{N}(0,\Sigma)
$
is a Gaussian random vector with symmetric positive-definite covariance matrix $\Sigma\in\mathbb{R}^{K\times K}$.

Under \eqref{eq:data-model}, the conditional distribution of $y$ given $q$ admits the likelihood
\[
p(y\mid q)\propto \exp\bigl(-\Phi^y(q)\bigr),
\]
where the data misfit functional is defined by
\begin{equation}
\Phi^y(q)
:=
\frac12 \|\mathcal{G}(q)-y\|_{\Sigma}^2
=
\frac12 \|\Sigma^{-1/2}(\mathcal{G}(q)-y)\|_2^2.
\tag{2.5}\label{eq:potential}
\end{equation}
To avoid confusion, the $\Phi^y$ will be simplified as $\Phi$ in the following paper.

Given a prior probability measure $\mu_{\mathrm{pr}}$ on $X$, the corresponding posterior measure $\mu^y$ is formally given by Bayes' rule
\begin{equation}
\frac{d\mu^y}{d\mu_{\mathrm{pr}}}(q)
=
\frac{1}{Z(y)}
\exp\bigl(-\Phi(q)\bigr),
\tag{2.6}\label{eq:bayes-rule}
\end{equation}
where the normalization constant is

$$Z(y):=\int_{X}\exp\bigl(-\Phi(q)\bigr)\,d\mu_{\mathrm{pr}}(q).
\label{eq:normalizing-constant}$$

In this section, we take $\mu_{\mathrm{pr}}=\mu_0$, where $\mu_0$ is a periodic Gaussian reference measure on $X$. The PHG prior will be defined as a weighted hybrid prior with respect to $\mu_0$ in the next section.

\begin{remark}
If the observed data consists of phased far-field measurements, it suffices to replace \eqref{eq:data-model} by
\[
\mathcal{G}(q)
=
\bigl(
\mathcal{F}(q)(\hat{x}_1),\dots,\mathcal{F}(q)(\hat{x}_K)
\bigr)^T,
\]
with the corresponding likelihood defined on the observation space.
\end{remark}

\subsection{Periodic Gaussian prior and well-posedness}
\par Let us consider the observation space  \(Y=\mathbb{R}^K\) and work on the parameter space  \(X=C_{\mathrm{per}}^{2,\alpha}(S^1)\), which we identify with the Banach space of $2\pi$-periodic functions in $C^{2,\alpha}([0,2\pi])$ with Hölder exponent \(\alpha\in(0,1)\). A general approach to ensure \(q(\theta)\in C_{\mathrm{per}}^{2,\alpha}\) is to impose a Gaussian prior on its second derivative (e.g., \(q''\sim\mathcal{GP}(0,(-\Delta)^{-1})\) \cite{bui2014analysis}) with periodic boundary conditions, while we adopt a Gaussian reference measure approach for the function $q$ with a periodic squared-exponential covariance. Let the prior distribution for the log-radial function $q$ be given by \cite{seeger2004gaussian}
\begin{equation*}
q\sim \mathcal{GP}(m(\theta), k(\theta,\theta')),\quad \theta,\theta' \in [0,2\pi],
\end{equation*}
where $\mathcal{GP}$ denotes a Gaussian process with mean function $m(\theta)$ and covariance kernel $k(\theta,\theta')$. For simplicity, we set $m(\theta) \equiv 0$. The periodic squared exponential covariance kernel is given by
\begin{equation}\label{eq:seckernel}
k(\theta,\theta') = \sigma^2 \exp\!\left( 
-\frac{2 \sin^{2}\!\left( \pi |\theta-\theta'|/p \right)}{l^{2}}
\right)\,
\tag{2.7}
\end{equation}
where $\sigma^2 > 0$ controls the variance amplitude, and $\ell > 0$ controls the characteristic length-scale, $p$ is the period length. 
\begin{proposition}[Continuity of higher-order derivatives]
Let \(q\) be the Gaussian random field with periodic squared-exponential kernel \eqref{eq:seckernel}. Then \( q \) belongs to $C_{\mathrm{per}}^{2,\alpha}(S^1)$ almost surely. More precisely, $q''\in C_{\mathrm{per}}^{0,\alpha}(S^1)$ almost surely.

\end{proposition}

\begin{proof}
According to the analysis in \cite{seeger2004gaussian}, if \( k(\theta, \theta') \) is the covariance kernel of \( q \) and has mixed fourth derivative, then the covariance kernel of \( q''\) is given by
\[
\text{Cov}(q''(\theta), q''(\theta')) = \frac{\partial^4 k(\theta, \theta')}{\partial \theta^2 \partial \theta'^2} = k^{(2,2)}(\theta, \theta').
\]
Let \( z = \frac{\pi(\theta - \theta')}{p} \), explicit computation yields
\[
k^{(2,2)}(\theta, \theta') = k(\theta, \theta') \left[ A_1 \sin^8\left(z\right) + A_2 \sin^6\left(z\right)+A_3 \sin^4\left(z\right)+A_4 \sin^2\left(z\right) + A_5 \right],
\]
where the constants $A_i$$(i=1:5)$ depend on $p$ and $\ell$, let \[
G(z) = A_1 \sin^8 z + A_2 \sin^6 z + A_3 \sin^4 z + A_4 \sin^2 z + A_5,
\] 
and using $|\sin(z)| < |z|$, we have
\[
|G(z) - A_5| = \left| A_1 \sin^8 z + A_2 \sin^6 z + A_3 \sin^4 z + A_4 \sin^2 z \right|
\leq (|A_1| + |A_2| + |A_3| + |A_4|) z^2 = C_1 |\theta - \theta'|^2.
\]
Moreover, we have
\[
|k(\theta, \theta') - k(\theta,\theta)| = \left| \sigma^2 - \sigma^2\exp\left( -\frac{2\sin^2 z}{l^2} \right) \right|
\leq \frac{2\sigma^2}{l^2} \sin^2 z \leq C_2 |\theta - \theta'|^2.
\]
Based on the above analysis, the following estimate holds
\begin{align*}
& \left| k^{(2,2)}(\theta, \theta') - k^{(2,2)}(\theta,\theta) \right| \\
=\, & \left| k(\theta, \theta') G(z) - k(\theta,\theta) A_5 \right| \\
=\,& \left| k(\theta,\theta) \right| \left| G(z) - A_5 \right| 
+ \left| A_5 \right| \left| k(\theta, \theta') - k(\theta,\theta) \right|
+ \left| k(\theta, \theta') - k(\theta,\theta) \right| \left| G(z) - A_5 \right|\\
\leq\, & C_1 |\theta - \theta'|^2 + |A_5| \cdot C_2 |\theta - \theta'|^2+C_1 C_2|\theta - \theta'|^4\\
\leq & C |\theta - \theta'|^2,
\end{align*}
where $C$ is an explicit constant depending on $p,l$. This gives the second moment estimation
\[
\mathbb{E}|q''(\theta) - q''(\theta')|^2 \leq C |\theta-\theta'|^2.
\]

For any \(m\in \mathbb N\), the increment \(q''(\theta)-q''(\theta')\) is a centered Gaussian random variable. Hence there exists a constant \(C_m>0\) (Corollary 6.8 in \cite{stuart2010inverse}), such that
\[
\mathbb E\bigl|q''(\theta)-q''(\theta')\bigr|^{2m}
\le
C_m\Bigl(\mathbb E\bigl|q''(\theta)-q''(\theta')\bigr|^2\Bigr)^m.
\]
Using the estimate
\[
\mathbb E\bigl|q''(\theta)-q''(\theta')\bigr|^2 \le C|\theta-\theta'|^2,
\]
we obtain
\[
\mathbb E\bigl|q''(\theta)-q''(\theta')\bigr|^{2m}
\le
C_m |\theta-\theta'|^{2m}.
\]

For any \(\alpha<1\), and choose \(m\in\mathbb N\) sufficiently large so that
$
\alpha<1-\frac{1}{2m},
$
the Kolmogorov continuity theorem implies that \(q''\in C^{0,\alpha}(S^1)\) almost surely. Since the covariance kernel \eqref{eq:seckernel} is periodic, the sample paths are periodic, and therefore
$
q\in C_{\mathrm{per}}^{2,\alpha}(S^1)
$ almost surely.
\end{proof}

\begin{proposition}\label{prop:stability}
Let $ \mu_{0}$ be the Gaussian reference measure induced by the periodic kernel \eqref{eq:seckernel}. It follows that $\mu_{0}(X) = 1$ for $0 < \alpha < 1$.
\end{proposition}
\begin{proof}
The above analysis shows that $q''$ is almost surely in the space $C_{\mathrm{per}}^{0,\alpha}(S^1)$ for any $0<\alpha<1$ .Then, it yields that $q$ is almost surely in $X=C_{\mathrm{per}}^{2,\alpha}(S^1)$ and $\mu_{0}(X)=1$. 
\end{proof}

Then, we give the well-posedness results of the periodic squared exponential prior \eqref{eq:seckernel}, which mainly based on an analysis of the forward map 
$\mathcal{G}$. Moreover, under additional assumptions, the posterior measure can be well approximated by a finite-dimensional representation (Theorem 4.10 in \cite{stuart2010inverse}).
These results are crucial for establishing the well-posedness proof of the PHG prior in ~\Cref{sec:wellposed}.
\begin{theorem}\label{theo:2.1}
Let $q\sim \mu_{0}$  be given by \eqref{eq:seckernel}, then the following properties hold
\begin{enumerate}[label=(\roman*), leftmargin=2em]
\item The posterior measure $\mu^{y}$ is absolutely continuous with respect to $\mu_{0}$, and the Radon--Nikodym derivative is given by
\begin{equation*}
\frac{d\mu^{y}}{d\mu_{0}}
= \frac{1}{Z}\exp\!\left(-\tfrac{1}{2}\,\|y-\mathcal{G}(q)\|_{\Sigma}^{2}\right),
\end{equation*}
where
\[
Z=\int_{X}\exp\!\left(-\tfrac{1}{2}\,\|y-\mathcal{G}(q)\|_{\Sigma}^{2}\right)\,\mu_{0}(dq).
\]
\item The posterior measure $\mu^y$ is a well-defined probability measure on 
$X=C_{\mathrm{per}}^{2,\alpha}(S^1)$.
\item The posterior measure $\mu^y$ is Lipschitz continuous with respect to the data $y$ in the Hellinger distance.
\end{enumerate}
\end{theorem}

\begin{proof}
First, $\mu_{0}(X)=1$ due to Proposition \ref{prop:stability}. Then,
\Cref{theo:2.1} are direct consequences of the fact that satisfies Assumption A.1 in \cite{bui2014analysis} and 
applies Theorem A.2 in \cite{bui2014analysis} . The proof is finished.
\end{proof}
To prepare for the persistence-based prior construction in ~\Cref{sec:topo-prior}, we introduce the generic admissible set on which the one-dimensional persistence pairing is well defined and free of degeneracies.
\begin{equation}\label{eq:Xadm}
X_{\mathrm{adm}}
:=
\{q\in X:\ q \text{ is a periodic Morse function on } S^1
\text{ and its critical values are pairwise distinct}\}.
\tag{2.8}
\end{equation}

\begin{assumption}[Admissibility assumption]\label{ass:adm}
The set \(X_{\mathrm{adm}}\) is Borel measurable and has full \(\mu_0\)-measure
\[
\mu_0(X_{\mathrm{adm}})=1.
\]
\end{assumption}

This assumption is used only to ensure that the persistence-based quantities introduced below are well defined for \(\mu_0\)-almost every realization. For the periodic squared-exponential Gaussian reference prior considered here, it is natural in view of the generic Morse-type nondegeneracy expected for smooth Gaussian sample paths, although we do not pursue the technical proof in the present paper.

\section{Persistent-Homology-Gaussian prior}\label{sec:topo-prior}
The present section introduces the construction of the PHG prior, which is derived from the Gaussian reference measure \eqref{eq:seckernel} with persistence-based regularization for the log-radial function. Firstly, some basic definitions used for PHG prior have been given in \Cref{def:filtration}
and \Cref{def:persist}. Then, the weighted persistence-based penalty are also defined in \Cref{def:toppenal}, and PHG prior has been constructed as a Gibbs tilt of periodic Gaussian prior in \Cref{def:PHG}.

\subsection{Sublevel and lower-star filtrations}
\label{def:filtration}
Let $M$ be a triangulable topological space, and let $f:M\to\mathbb R$ be a continuous tame function. For each threshold $c\in\mathbb R$, the sublevel set is defined by
\[
S_c=\{x\in M\mid f(x)\le c\}.
\]
The family $\{S_c\}_{c\in\mathbb R}$ forms a nested filtration, and persistent homology tracks the birth and death of homological features along this filtration. For tame functions, the homology changes only at finitely many filtration values.

To compute this filtration in practice, the continuous setting is replaced by a simplicial discretization. Specifically, the space $M$ is approximated by a simplicial complex $\mathcal K$ with vertex set $V$, and the function $f$ is specified at the vertices and extended to simplices by the lower-star rule
\[
f(\sigma)=\max_{v\in\sigma} f(v).
\]
This extension induces the lower-star filtration,
\[
\mathcal K_c=\{\sigma\in\mathcal K\mid f(\sigma)\le c\},\qquad c\in\mathbb R,
\]
which provides the natural simplicial counterpart of the continuous sublevel-set filtration \cite{edelsbrunner2008persistent}.

Persistence is then computed by pairing simplices in the filtration. When a simplex $\sigma$ enters the filtration, it may create a new homology class, whereas a later simplex $\tau$ may destroy an existing one. The persistence pair $(\sigma,\tau)$ records the birth and death of the corresponding topological feature, and its persistence is given by
\begin{equation}
   \mathrm{pers}(\sigma,\tau)=f(\tau)-f(\sigma).\tag{3.1}\label{eq:pers}
\end{equation}

The persistence diagram consists of all such birth--death pairs and all essential classes with death time $c=+\infty$. In the sequel, the ambient space is the unit circle \(S^1\), and the log-radial function \(q\) belongs to the periodic Hölder space \(C_{\mathrm{per}}^{2,\alpha}(S^1)\). Since \(X=C_{\mathrm{per}}^{2,\alpha}(S^1)\subset C^0(S^1)\), the sublevel-set filtration is well defined for every \(q\in X\). To define the persistence-based penalty introduced below, we restrict attention to an admissible subset \(X_{\mathrm{adm}}\subset X\) in \eqref{eq:Xadm} on which the one-dimensional pairing structure is unambiguous.
\subsection{Persistence pairing of the discretized log-radial function}\label{def:persist}
In this subsection, we will consider the discrete approximation
\(q_h\) of the log-radial function\(q:S^1\to\mathbb R\), and its associated persistence pairing is obtained by substituting $q_h$ into equation \eqref{eq:pers} in the one-dimensional periodic profile. Here, we view $S^1$ as the interval $[0,2\pi)$ with endpoints identified. Consider an ordered set of grid points $0 \le x_1 < \cdots < x_m < 2\pi$ arranged cyclically on the interval $[0, 2\pi)$, and let $\mathcal{K}_h$ denote the corresponding one-dimensional simplicial complex whose vertices $\{v_i\}_{i=1}^m$ are mapped to the points $\{x_i\}_{i=1}^m$, with edges connecting consecutive vertices $(v_i, v_{i+1})$ for $1 \le i \le m-1$ together with the periodic boundary edge $(v_m, v_1)$. Then, we have \(q_h(v_i)=q(x_i)\) on the vertices. By abuse of notation, \(q_h\) is also presented as the associated periodic piecewise linear interpolant on \(S^1\). To induce the sublevel-set filtration, we assign function values to edges by the maximum rule, where indices are understood cyclically.
\[
q_h((v_i,v_{i+1}))=\max\{q_h(v_i),q_h(v_{i+1})\}.
\]
We describe the generic case in which the sampled nodal values are distinct. The distinctness assumption is imposed here only for clarity of exposition, in the discrete lower-star setting, it can be removed by adopting a deterministic pairing rule for repeated nodal values \cite{zheng2015application}. In this case, the finite persistence pairs relevant here arise from \(H_0\), and we write them as \((x_k,x_\ell)\), where \(x_k\) is the birth location and \(x_\ell\) is the corresponding death location.

\begin{remark}In the periodic setting, a single essential \(H_1\) class is created when the last edge enters the filtration and the loop becomes complete. Its birth time coincides with the terminal filtration value, which is the same filtration value at which the final finite $H_0$ class dies. Thus the essential $H_1$ class carries no additional information beyond the final $H_0$ merge and is omitted from the penalty \eqref{eq:R-def}.
\end{remark}

Assuming that the sampled values \(q(x_i)\) are distinct, each finite \(H_0\) persistence pair is uniquely determined by a birth vertex and a death vertex. If \(x_k\) creates a connected component and this component first merges with an older one at filtration level \(q_h(x_\ell)\), then \((x_k,x_\ell)\) is a discretized persistence pair with lifetime
\[
\mathrm{pers}(x_k,x_\ell)=q_h(x_\ell)-q_h(x_k).
\]
The essential class representing the global minimum, which persists indefinitely in the sublevel filtration, is paired with $c=+\infty$. To obtain a symmetric descriptor, we also consider the superlevel information by applying the same construction to $-q_h$. Let $P_+(q_h)$ collect all finite persistence pairs for $q_h$, and let $P_-(q_h)$ be the corresponding collection for $-q_h$. Define the persistence functional
\begin{equation}
    \|q_h\|_{\text{per}} = \sum_{(v_k, v_\ell) \in P_+(q_h)} |q_h(v_\ell) - q_h(v_k)| + \sum_{(v_k, v_\ell) \in P_-(q_h)} |q_h(v_\ell) - q_h(v_k)|.\tag{3.2}\label{eq:q-per}
\end{equation}

\begin{remark}[Relation to total variation]\label{rem:tv_discrete}
In one dimension, persistence lifetimes provide a natural decomposition of oscillatory variation. If \(x_{i_0},\dots,x_{i_{r-1}}\) denote the discrete extrema of \(q_h\), listed in cyclic order around \(S^1\), and \(x_{i_r}=x_{i_0}\) by cyclic convention, then
\[
TV(q_h)=\sum_{l=0}^{r-1}|q_h(x_{i_{l+1}})-q_h(x_{i_l})|
      =\|q_h\|_{\mathrm{per}}+\max_{i,j}|q_h(v_i)-q_h(v_j)|.
\]
Thus, \(\|q_h\|_{\mathrm{per}}\) captures the oscillatory part of the variation, while the remaining global range accounts for the overall amplitude. This discrete identity serves as an intuitive guide for the continuous construction below.
\end{remark}

\begin{remark}[Relation between \(q\) and its discrete approximation \(q_h\)]
For a periodic function \(q\in C_{\mathrm{per}}^{2,\alpha}(S^1)\), the persistence pairs used in computation are constructed from its mesh-based approximation \(q_h\). The use of this discretization is supported by the stability of persistence diagrams \cite{cohen2005stability}: if \(D(\cdot)\) denotes the persistence diagram of the sublevel-set filtration and \(d_B\) is the bottleneck distance, then
\[
d_B\bigl(D(q),D(q_h)\bigr)\le \|q-q_h\|_\infty.
\]
As the mesh is refined \((h\to 0)\), one has \(\|q-q_h\|_{\infty}\to 0\), and hence \(D(q_h)\to D(q)\) in bottleneck distance. Since \(q_h\) is understood as the periodic piecewise linear interpolant introduced above, its lower-star persistence agrees with the sublevel-set persistence of the corresponding PL function. Thus the mesh-based persistence diagrams provide a stable discrete approximation of the continuous sublevel-set persistence. To pass from diagram convergence to the convergence of the lifetime-based penalty, an additional consistency result for the one-dimensional pairing structure is needed; this will be stated below.
\end{remark}

\subsection{Weighted persistence-based penalty }\label{def:toppenal}
The one-dimensional pairing structure described above gives rise to a family of lifetimes associated with paired local extrema of \(q\).
Short lifetimes typically reflect small-amplitude oscillations, whereas long lifetimes correspond to more dominant features across scales. Motivated by this structure, we introduce a weighted persistence-based regularization by aggregating these lifetimes. 
By Assumption \eqref{ass:adm}, the admissible set \(X_{\mathrm{adm}}\subset X\) has full \(\mu_0\)-measure. We therefore define the persistence-based penalty on \(X_{\mathrm{adm}}\), where the one-dimensional periodic pairing is well defined without ambiguity, and then relate it to its mesh-based approximation used in computation.

For \(q\in X_{\mathrm{adm}}\), let \(P_+(q)\) and \(P_-(q)\) denote the finite \(H_0\) persistence pairs arising from the sublevel filtrations of \(q\) and \(-q\), respectively. To incorporate the global range term identified in Remark~\ref{rem:tv_discrete}, we adjoin the distinguished pair \((x_{\min},x_{\max})\), where \(x_{\min}\) and \(x_{\max}\) denote any locations of a global minimum and a global maximum of \(q\), respectively. Define

\[
\widetilde P(q):=P_{+}(q)\cup P_{-}(q)\cup \{(x_{\min},x_{\max})\}.
\]
For each pair \(p=(x,\tilde x)\in \widetilde P(q)\), define its associated lifetime by
\[
\ell_p(q):=\bigl|q(\tilde x)-q(x)\bigr|.
\]
Motivated by the discrete identity in Remark~\ref{rem:tv_discrete}, we now record the corresponding continuous decomposition.
\begin{proposition}[Relation to total variation in the continuous setting]\label{prop:tv_cont}
Let \(q\in X_{\mathrm{adm}}\), and
$
x_0,\dots,x_{r-1}\in S^1$
be the local extrema of \(q\), listed in cyclic order around \(S^1\), and set \(x_r=x_0\). Then
\[
TV(q)=\int_{S^1}|q'(\theta)|\,d\theta
     =\sum_{i=0}^{r-1}|q(x_{i+1})-q(x_i)|.
\]
Moreover, the one-dimensional persistence pairing described above yields the decomposition
\[
TV(q)=\sum_{p\in \bar P(q)}\ell_p(q),
\]
The discrete non-periodic version of this persistence--TV relation was established in \cite{zheng2015application}. The present one-dimensional periodic continuous setting can be viewed as the corresponding extension; since our subsequent analysis only uses the mesh-based formulation, we do not reproduce the continuous proof here.
\end{proposition}

Whenever the resulting weighted sum is well defined, we define the persistence-based penalty
\begin{equation}\tag{3.3}\label{eq:R-def}
R(q)
:=
\sum_{p\in \widetilde P(q)} \alpha_p(q)\,\ell_p(q).
\end{equation}
A simple lifetime-dependent weighting strategy is
\begin{equation}\tag{3.4}\label{eq:alpha-basic}
\alpha_p(q)
=
\frac{\tau}{1+\beta\,\ell_p(q)}
\qquad \tau>0,\ \beta>0.
\end{equation}
This choice assigns larger per-unit weights to small lifetimes, thereby discouraging minor oscillations, while reducing the relative weight assigned to large lifetimes so as to preserve dominant features.

If an additional ordering of pairs is available, for instance from a merge-tree representation, one may also consider the order-dependent variant
\begin{equation}\tag{3.5}\label{eq:alpha-order}
\alpha_p(q)
=
(\kappa_p+1)\tau\,\frac{1}{1+\beta\,\ell_p(q)}
\end{equation}
where \(\kappa_p\) is the order index associated with the pair \(p\), see \cite{zheng2015application} for related order-based constructions. For the order-weight strategy \eqref{eq:alpha-order}, we assume the order factor is uniformly bounded, which guarantees the local boundedness required by ~\Cref{sec:wellposed}. In all cases, the role of $R(q)$ is to encode prior preference for reconstructions whose
topological variability is controlled through persistence lifetimes.

\begin{proposition}[Consistency of the mesh-based penalty]\label{prop:COMEP}
Let \(q\in X_{\mathrm{adm}}\), and let \(q_h\) be its periodic piecewise linear interpolant on a quasi-uniform mesh. Define the mesh-based penalty \(R_h(q_h)\) from the discrete periodic pairing rule and the lifetime-dependent weight \eqref{eq:alpha-basic}. Then, for sufficiently fine meshes, the discrete pairing of \(q_h\) is combinatorially consistent with the continuous one-dimensional pairing of \(q\), and
\[
R_h(q_h)\to R(q), \qquad h\to 0.
\]
The same conclusion holds for the order-dependent weight \eqref{eq:alpha-order}, provided the associated order index is eventually preserved by the discrete pairing.

\end{proposition}
\begin{proof}
Since \(q_h\) is the periodic piecewise linear interpolant of 
\(q\) on a quasi-uniform mesh and 
\(q\in C_{\mathrm{per}}^{2,\alpha}(S^1)\), standard interpolation estimates imply
\[
\|q_h-q\|_{\infty}\to 0 \qquad \text{as } h\to 0.
\]
Since \(q\in X_{\mathrm{adm}}\), its critical points are isolated and finite in number on \(S^1\). Hence there exist pairwise disjoint neighborhoods \(U_1,\dots,U_N\) of these critical points such that \(q'\) changes sign exactly once in each \(U_j\), while
\[
|q'(\theta)|\ge \eta>0
\qquad\text{for all }\theta\in S^1\setminus \bigcup_{j=1}^N U_j .
\]
By continuity of \(q'\), the sign of \(q'\) is constant on each connected component of \(S^1\setminus \bigcup_j U_j\), and the mean value theorem implies that, for sufficiently fine meshes, the nodal values of \(q_h\) are strictly monotone on each such component. Likewise, on each \(U_j\), the unique sign change of \(q'\) yields exactly one extremum knot of the same type for the sampled profile of \(q_h\). Therefore, for all sufficiently fine meshes, the discrete extrema of \(q_h\) are in one-to-one correspondence with the critical points of \(q\), with the same cyclic order and no additional extrema elsewhere. Since the critical values of \(q\) are pairwise distinct, the corresponding extremal values of \(q_h\) preserve the same strict ordering for all sufficiently small \(h\). In one dimension, the \(0\)-dimensional persistence pairing is determined by the cyclic order of local extrema together with the elder rule \cite{zheng2015application}, hence the sublevel pairing of \(q_h\) agrees combinatorially with the continuous pairing of \(q\). The same argument applies to the superlevel pairing after replacing \(q\) by \(-q\). Thus the augmented pair collections \(\widetilde P_h(q_h)\) and \(\widetilde P(q)\) are in bijection for all sufficiently fine meshes.

Let \(p_h=(x_h,\tilde x_h)\in \widetilde P_h(q_h)\) be the discrete pair corresponding to \(p=(x,\tilde x)\in \widetilde P(q)\). By the above correspondence, \(x_h\) and \(\tilde x_h\) lie in the neighborhoods of \(x\) and \(\tilde x\), respectively, and mesh refinement implies \(x_h\to x\) and \(\tilde x_h\to \tilde x\). Hence, using also \(\|q_h-q\|_{\infty}\to 0\),
\[
\ell_{p_h}(q_h)=|q_h(\tilde x_h)-q_h(x_h)| \to |q(\tilde x)-q(x)|=\ell_p(q).
\]
For the lifetime-dependent weight \eqref{eq:alpha-basic}, the map
$\ell \mapsto \alpha(\ell)\,\ell$
is continuous on \([0,\infty)\). Hence
\[
\alpha\bigl(\ell_{p_h}(q_h)\bigr)\,\ell_{p_h}(q_h)
\;\to\;
\alpha\bigl(\ell_p(q)\bigr)\,\ell_p(q).
\]
Since the number of pairs is finite and eventually constant, summing over the corresponding pairs yields
\[
R_h(q_h)
=
\sum_{p_h\in \bar P_h(q_h)}
\alpha\bigl(\ell_{p_h}(q_h)\bigr)\,\ell_{p_h}(q_h)
\;\to\;
\sum_{p\in \bar P(q)}
\alpha\bigl(\ell_p(q)\bigr)\,\ell_p(q)
=
R(q).
\]

For the order-dependent weight \eqref{eq:alpha-order}, the same conclusion holds provided the associated order index is determined by the pairing structure and is therefore eventually preserved under the above combinatorial consistency.
\end{proof}

\subsection{PHG prior as a Gibbs tilt of periodic Gaussian prior}\label{def:PHG}
In this subsection, we give the formal Gibbs-tilt representation of the PHG prior relative to the periodic Gaussian reference measure and derive the corresponding posterior density. Let \(\mu_0\) denote the Gaussian reference measure on \(X=C_{\mathrm{per}}^{2,\alpha}(S^1)\) introduced in ~\Cref{se2}. By Assumption \ref{ass:adm}, the persistence-based penalty is defined on a full-\(\mu_0\)-measure admissible set \(X_{\mathrm{adm}}\subset X\). By abuse of notation, we denote by \(R\) a measurable extension of the persistence-based penalty from \(X_{\mathrm{adm}}\) to \(X\), the resulting Gibbs-tilted measures are independent of the particular choice of extension, since \(X\setminus X_{\mathrm{adm}}\) is \(\mu_0\)-null. We formally define the PHG prior \(\mu_{\mathrm{pr}}\) as the Gibbs tilt of \(\mu_0\),
\[
\frac{d\mu_{\mathrm{pr}}}{d\mu_0}(q)
=
\frac{1}{Z_R}\exp\bigl(-R(q)\bigr),
\qquad
Z_R:=\int \exp\bigl(-R(q)\bigr)\,d\mu_0(q),
\tag{3.6}\label{eq:gibbs-tilt-prior}
\]
provided that the extension \(R:X\to\mathbb R\) is measurable and that \(0<Z_R<\infty\).

Combining \eqref{eq:gibbs-tilt-prior} with the likelihood potential \(\Phi^y(q)\) from ~\Cref{se2}, Bayes' formula yields the posterior measure
\[
\frac{d\mu^{y}}{d\mu_0}(q)
=
\frac{1}{Z(y)}\exp\bigl(-\Phi^y(q)-R(q)\bigr),
\tag{3.7}\label{eq:gibbs-tilt-posterior}
\]
where
\[
Z(y):=\int \exp\bigl(-\Phi^y(q)-R(q)\bigr)\,d\mu_0(q).
\]

The next section develops a general Gibbs-tilt analysis for posterior measures of the form \eqref{eq:gibbs-tilt-posterior}. We first identify verifiable assumptions on the persistence-based penalty \(R\) that ensure well-posedness and convergence of finite-dimensional approximations, and then verify that the persistence-based penalty \eqref{eq:R-def} satisfies these assumptions in the present setting.

\section{Theoretical analysis of PHG prior}
\label{sec:wellposed}
This section presents the theoretical analysis of the PHG prior constructed in Section~\ref{sec:topo-prior}. We adopt the Gibbs-tilted posterior measure \eqref{eq:gibbs-tilt-posterior} and establish its well-posedness in the sense of \cite{latz2023bayesian}, including (i) stability with respect to perturbations in the data under several standard probability metrics, and (ii) convergence of finite-dimensional approximations. Our analysis proceeds in two steps. First, we formulate a general Gibbs-tilt framework on a full-measure admissible subset \(X_{\mathrm{adm}}\subset X\) and state assumptions on an additional functional \(R\) defined on \(X_{\mathrm{adm}}\) that ensure well-posedness of the posterior. Second, we verify that the persistence-based penalty $R$ defined in \eqref{eq:R-def} 
satisfies these assumptions by establishing its lower boundedness, local boundedness, and continuity on \(X_{\mathrm{adm}}\).

\subsection{Well-posedness under the general Gibbs-tilt framework}
We first establish a well-posedness result for Gibbs-tilted posteriors with an additional regularization functional \(R\). Let \((X,\|\cdot\|_X)\) be the parameter space for the log-radial function \(q\), and let \(\mu_0\) be the Gaussian reference measure introduced in ~\Cref{se2}. Let \(X_{\mathrm{adm}}\subset X\) be a measurable admissible subset satisfying \Cref{ass:adm} , and let $
R:X_{\mathrm{adm}}\to\mathbb R
$ be an additional functional. By abuse of notation, we again denote by \(R\) a measurable extension of this functional to \(X\). We consider the Gibbs-tilted prior and posterior (cf.\ \eqref{eq:gibbs-tilt-prior}--\eqref{eq:gibbs-tilt-posterior}), where
$
Z(y):=\int_X \exp\bigl(-\Phi^y(q)-R(q)\bigr)\,d\mu_0(q)
$
is the normalizing constant. Assumption A.1 in \cite{bui2014analysis} provides the required measurability and continuity properties of
the likelihood potential $\Phi^y$. Hence, the remaining task is to impose mild growth and
regularity conditions on $R$ ensuring that the posterior is well-defined and stable with respect to
perturbations in the data $y$.

Classical Bayesian well-posedness is often formulated as local Lipschitz continuity of the map
$y\mapsto\mu^y$ in the Hellinger distance \cite{stuart2010inverse}. Following the generalized
viewpoint of \cite{latz2023bayesian}, we adopt the notion of $(P,d)$-well-posedness, which provides a
unified route to stability under a broader class of probability metrics. The key additional requirement beyond ~\Cref{se2} is that \(R\) be bounded below, locally bounded on bounded subsets of \(X_{\mathrm{adm}}\), and continuous on \(X_{\mathrm{adm}}\).


\begin{assumption}\label{ass:A2}
Let \(X_{\mathrm{adm}}\subset X\) be the admissible set from Assumption~\eqref{ass:adm}. The functional
$
R:X_{\mathrm{adm}}\to\mathbb R
$
satisfies:
\begin{enumerate}[label=(\roman*), leftmargin=*, widest=(iii)]
\item \textbf{(Lower bound)} \(R\) is bounded from below on \(X_{\mathrm{adm}}\). Without loss of generality, we assume \(R(q)\ge 0\) for all \(q\in X_{\mathrm{adm}}\).
\item \textbf{(Local boundedness)} For every $r>0$, there exists $K=K(r)>0$ such that
$R(q)\le K(r)$ for all $q\in X_{\mathrm{adm}}$ with $\|q\|_X\le r$.
\item \textbf{(Continuity)} $R$ is continuous on $X_{\mathrm{adm}}$.
\end{enumerate}
\end{assumption}





\begin{theorem}\label{thm:wellposed-hybrid}
Suppose Assumption A.1 in \cite{bui2014analysis} holds for the likelihood potential $\Phi^y$ and $R:X_{\mathrm{adm}}\to\mathbb R$ satisfies Assumption~\ref{ass:A2} .
Then the posterior \(\mu^y\) defined by \eqref{eq:gibbs-tilt-posterior} is \((\mathcal P,d)\)-well-posed in the sense of \cite{latz2023bayesian}. In particular, it is well-posed under the weak, Hellinger, and total variation metrics, and also under the Wasserstein-\(p\) metric whenever the corresponding moment conditions hold.

\end{theorem}
\begin{proof}
Under Assumption A.1 in \cite{bui2014analysis}, the likelihood potential $\Phi^y$ satisfies the regularity
conditions required in \cite{latz2023bayesian}. 
Assumption~\ref{ass:adm} ensures that \(X_{\mathrm{adm}}\) has full \(\mu_0\)-measure, so the Gibbs factor \(\exp(-R(q))\) is defined \(\mu_0\)-almost surely, and any measurable extension of \(R\) from \(X_{\mathrm{adm}}\) to \(X\) yields the same Gibbs-tilted measure. Assumption~\ref{ass:A2} provides the required lower-bound, boundedness, and continuity properties of \(R\) on \(X_{\mathrm{adm}}\). Therefore, the assumptions of \cite[Assumptions 3.5 and 3.10]{latz2023bayesian} apply to
the Gibbs-tilted posterior \eqref{eq:gibbs-tilt-posterior}. The claim follows by
\cite[Theorems 3.6 and 3.12]{latz2023bayesian}.
\end{proof}

\subsection{Finite-dimensional approximation}

For practical computation, we consider a finite-dimensional approximation of the posterior measure
$\mu^y$
\[
\frac{d \mu^y_{N_1, N_2}}{d \mu_0} \propto \exp \left( - \Phi_{N_1}(q;y) - R_{N_2}(q) \right),
\]
where $\Phi_{N_1}$ and $R_{N_2}$ denote suitable approximations of $\Phi$ and $R$, respectively.

\begin{theorem}\label{thm:fd-approx}
Assume that $\Phi$ and $\Phi_{N_1}$ satisfy Assumption A.1(i) in \cite{bui2014analysis} with constants uniform in $N_1$,
and that $R$ and $R_{N_2}$ satisfy Assumption~\ref{ass:A2}(i)--(ii) on \(X_{\mathrm{adm}}\) with constants uniform in $N_2$.
Assume also that for every $\epsilon>0$ there exist sequences $a_{N_1}(\epsilon)\to0$ and
$b_{N_2}(\epsilon)\to0$ such that $\mu_0(X_\epsilon)\ge 1-\epsilon$ for all $N_1,N_2\in\mathbb{N}$, where
\[
X_\epsilon = \left\{ q \in X_{\mathrm{adm}} \,\big|\, |\Phi(q;y) - \Phi_{N_1}(q;y)| \leq a_{N_1}(\epsilon), \;
|R(q) - R_{N_2}(q)| \leq b_{N_2}(\epsilon) \right\}.
\]
Then, as $N_1,N_2\to\infty$, we have convergence $\mu^y_{N_1,N_2}\to\mu^y$ in Hellinger distance.
Consequently, the convergence also holds in total variation and weak (e.g.\ Prokhorov) metrics.
Moreover, Wasserstein-$p$ convergence follows whenever the corresponding $p$-th moment conditions hold.
\end{theorem}

\begin{remark}
The Hellinger convergence statement can be obtained by adapting \cite[Theorem 2.3]{yao2016tv} to the
present Gibbs-tilt setting on the full-measure admissible set $X_{\mathrm{adm}}$. Standard inequalities imply that convergence in Hellinger distance
yields convergence in total variation, and hence also in weak metrics.
Wasserstein-$p$ convergence can be established by invoking \cite[Theorems 3.6 and 3.12]{latz2023bayesian}
under the appropriate moment conditions, we omit further details.
\end{remark}

\subsection{Well-posedness of PHG prior}
We now verify the lower-bound, local-boundedness, and continuity properties required by Assumption~\ref{ass:A2} on the admissible set \(X_{\mathrm{adm}}\) for the persistence-based penalty \(R\) defined in \eqref{eq:R-def}, with weights given by \eqref{eq:alpha-basic} or \eqref{eq:alpha-order}.


\begin{theorem}
    
\label{lem:topo-A2} 
The persistence-based penalty \(R\) defined in \eqref{eq:R-def} is bounded from below on \(X_{\mathrm{adm}}\). Moreover, for every \(q\in X_{\mathrm{adm}}\), \(R\) is locally bounded and continuous under the lifetime-dependent weight \eqref{eq:alpha-basic}. The same conclusion holds for the order-dependent weight \eqref{eq:alpha-order}, provided the order factor is uniformly bounded on bounded subsets of \(X\) and the associated order index is locally preserved under admissible perturbations.
\end{theorem}

\begin{proof}
\begin{itemize}
  \item[(i)] \emph{Lower bound.} Since each term is non-negative, the assumption is trivially satisfied.

  \item[(ii)]\emph{Boundedness on bounded sets.} We verify the corresponding boundedness estimate. From Proposition~\ref{prop:tv_cont}, we only need to consider the case of the TV regularization term and provide the corresponding proof. Fix \(r>0\), and let \(q\in X_{\mathrm{adm}}\) satisfy \(\|q\|_X\le r\), where \(X=C_{\mathrm{per}}^{2,\alpha}(S^1)\). Then
\[
TV(q)=\int_0^{2\pi}|q'(\theta)|\,d\theta
\le 2\pi \|q'\|_\infty
\le 2\pi \|q\|_X
\le 2\pi r.
\]

For the lifetime-dependent weight \eqref{eq:alpha-basic},
\[
\alpha_p(q)=\tau\frac{1}{1+\beta\ell_p(q)}\le \tau
\qquad \text{for all } p\in \bar P(q).
\]
Hence Proposition~\ref{prop:tv_cont} yields
\[
R(q)\le \tau\sum_{p\in \bar P(q)}\ell_p(q)
      = \tau\,TV(q)
      \le 2\pi\tau r.
\]
Thus \(R\) is bounded on bounded subsets of \(X_{\mathrm{adm}}\).

For the order-dependent weight \eqref{eq:alpha-order}, assume in addition that the order factor is uniformly bounded on bounded subsets of \(X\). Then, for \(\|q\|_X\le r\), there exists a constant \(C_r>0\) such that
\[
\kappa_p+1\le C_r
\qquad \text{for all } p\in \bar P(q).
\]
Consequently,
\[
\alpha_p(q)=(\kappa_p+1)\tau\frac{1}{1+\beta\ell_p(q)}
\le C_r\tau,
\]
and therefore
\[
R(q)\le C_r\tau\sum_{p\in \bar P(q)}\ell_p(q)
      = C_r\tau\,TV(q)
      \le 2\pi C_r\tau r.
\]
Thus \(R\) is also locally bounded in the order-dependent case.
\medskip
\noindent

  \item[(iii)] \emph{Continuity.}
For such a \(q\), write
\[
\bar P(q)=\{p_j=(x_j,\hat x_j)\}_{j=1}^N.
\]

Since \(q\in X_{\mathrm{adm}}\) is a one-dimensional periodic Morse function with pairwise critical values, all its critical points are isolated and admit unique continuations under sufficiently small \(C^1\)-perturbations, with the same type and cyclic order. Since convergence in \(X=C_{\mathrm{per}}^{2,\alpha}(S^1)\) implies convergence in \(C^1(S^1)\), there exists \(\delta>0\) such that for every admissible \(\tilde q\in X_{\mathrm{adm}}\) satisfying
\[
\|\tilde q-q\|_X<\delta,
\]
the local extrema of \(\tilde q\) are in one-to-one correspondence with those of \(q\). Combined with the one-dimensional elder-rule pairing described in \Cref{def:persist}, this implies that
\[
\bar P(\tilde q)=\{\tilde p_j=(\tilde x_j,\hat{\tilde x}_j)\}_{j=1}^N
\]
with the same number of pairs and the same pairing order as \(\bar P(q)\), and
\[
\tilde x_j\to x_j,
\qquad
\hat{\tilde x}_j\to \hat x_j
\qquad \text{as }\|\tilde q-q\|_X\to 0.
\]

For each \(j=1,\dots,N\), define
\[
\ell_{p_j}(q):=|q(\hat x_j)-q(x_j)|,
\qquad
\ell_{\tilde p_j}(\tilde q):=|\tilde q(\hat{\tilde x}_j)-\tilde q(\tilde x_j)|.
\]
Since \(\tilde q\to q\) in \(X\subset C^0(S^1)\), together with
\[
\tilde x_j\to x_j,
\qquad
\hat{\tilde x}_j\to \hat x_j,
\]
we obtain
\[
\ell_{\tilde p_j}(\tilde q)\to \ell_{p_j}(q),
\qquad j=1,\dots,N.
\]

For the lifetime-dependent weight \eqref{eq:alpha-basic}, the map
\[
\ell \mapsto \tau\frac{\ell}{1+\beta\ell},
\qquad \ell\ge 0,
\]
is continuous on \([0,\infty)\). Therefore,
\[
\tau\frac{\ell_{\tilde p_j}(\tilde q)}{1+\beta \ell_{\tilde p_j}(\tilde q)}
\to
\tau\frac{\ell_{p_j}(q)}{1+\beta \ell_{p_j}(q)},
\qquad j=1,\dots,N.
\]
Since the number of pairs is finite and constant for all admissible \(\tilde q\) with \(\|\tilde q-q\|_X<\delta\), summing over \(j\) yields
\[
R(\tilde q)
=
\sum_{j=1}^N
\tau\frac{\ell_{\tilde p_j}(\tilde q)}{1+\beta \ell_{\tilde p_j}(\tilde q)}
\to
\sum_{j=1}^N
\tau\frac{\ell_{p_j}(q)}{1+\beta \ell_{p_j}(q)}
=
R(q).
\]
Hence \(R\) is continuous at \(q\).

For the order-dependent weight \eqref{eq:alpha-order}, if the associated order index is locally preserved for all admissible \(\tilde q\in X_{\mathrm{adm}}\) with \(\|\tilde q-q\|_X<\delta\), then each factor \(\kappa_{p_j}+1\) remains constant for sufficiently small perturbations. The same argument therefore applies, and continuity follows also in this case.
\end{itemize}
\end{proof}

\section{Numerical experiments}\label{se5}
\subsection{Experimental setup}
In this section, we present several numerical examples to illustrate the effectiveness of the proposed method. The wavenumber is chosen as \(k=1\), and the incident wave is taken as a plane wave
\begin{equation*}
u^i(x) = e^{ikx \cdot d}, \quad d \in S^1.
\end{equation*}
\par To avoid the inverse crime, we use different forward solvers for data generation and inversion. Synthetic far-field data are generated by solving the scattering problem with the finite element method equipped with a perfectly matched layer (PML). In the inversion stage, the forward map is evaluated by a Nyström discretization of the boundary integral formulation with coupling parameter \(\xi=k\), and the far-field pattern is computed through equation \eqref{eq:2-3}. Relative additive Gaussian noise is added to the modulus of each observed far-field data as follows
\begin{equation*}
u_{obs}^{\infty}(\hat{x_i},d)
=| u^\infty(\hat{x_i},d)|+
\delta Z_i |u^\infty(\hat{x_i},d)|.
\end{equation*} Here, \(Z_i \overset{\mathrm{i.i.d.}}{\sim}\mathcal N(0,1)\) are independent standard Gaussian random variables, and \(\delta\) denotes the noise level. We collect $N_{\mathrm{obs}}=64$ far-field samples $u_{obs}^{\infty}(\hat{x_i},d)$ at uniformly spaced
observation directions $\hat{x_i}=(\cos\phi_i,\sin\phi_i)$ with
$\phi_i=\tfrac{2\pi(i-1)}{N_{\mathrm{obs}}}$, $i=1,\dots,N_{\mathrm{obs}}$.
The boundary is represented by a radial function $\rho(\theta)=exp(q)$ discretized on the same
$N_{\mathrm{obs}}$-point angular grid. The data vector \(y\in\mathbb R^{N_{\mathrm{obs}}}\) and the discretized parameter vector \((q(\theta_i))_{i=1}^{N_{\mathrm{obs}}}\in\mathbb R^{N_{\mathrm{obs}}}\) are both $N_{\mathrm{obs}}$-dimensional. We consider two representative noise levels, \(\delta=5\%\) and \(\delta=10\%\), to verify the robustness of the method under moderate and relatively high noise. Three obstacles for testing are considered

\begin{table}[h]
    \centering
    \renewcommand{\arraystretch}{1.5}
    \begin{tabular}{l c}
        \hline
        \textbf{Obstacle} & \textbf{Parametrized boundary} \\
        \hline
        Kite & $(x_1, x_2) = (\cos\theta + 0.65\cos 2\theta - 0.65, 1.5\sin\theta)+x_c$ \\
        
        Rectangle & $(x_1, x_2) = (\cos(\theta)/\max(|\cos(\theta)|,|\sin(\theta)|),\sin(\theta)/\max(|\cos(\theta)|,|\sin(\theta)|))+x_c$\\
    5-star & $(x_1, x_2) = ((1+0.3\sin5\theta)\cos\theta , (1+0.3\sin5\theta)\sin\theta)+x_c$ \\
        \hline
    \end{tabular}
    \caption{Parametrized boundary of the object obstacles}
\end{table}

\begin{remark}
For domains with corners, a uniform angular discretization may resolve the corners poorly, especially when their locations are unknown a priori. This makes such examples a useful test for the proposed PHG prior under noisy observations.

\end{remark}
We use the pCN Markov chain Monte Carlo algorithm to implement the sampling process (See \Cref{alg:pcn-MCMC}). Due to different observation apertures, the convergence rate of the MCMC algorithm may vary. Thus, The posterior mean is estimated from the final 5,000 samples of an 8,000-step Markov chain, following a burn-in period of 3,000 iterations (Such as for 5-star when $\delta =10\%$, we use 40000 iterations and a burn-in of 20000 samples). 

\begin{algorithm}[h]
\caption{The pCN MCMC algorithm with PHG prior}
\label{alg:pcn-MCMC}
\begin{algorithmic}[1]
\State \textbf{Inputs:} observed far-field data \(y\), step-size \(\beta_{PCN}\in(0,1)\), iterations \(N\).
\State \textbf{Prior on log-radial function:} \(q \sim \mathcal{GP}(0,k)\) with the periodic kernel \eqref{eq:seckernel}.
\State \textbf{Init:} set \(q^{0}\equiv 0\); define \(\rho^{0}=\exp(q^{0})\).
Compute \(g^{0}=\mathcal{G}(q^{0})\), \(\Phi^{0}=\Phi(q^{0};y)\), and \(R^{0}=R(q^{0})\).
\For{$j=0,1,\dots,N-1$}
  \State \textbf{pCN proposal in \(q\):} draw \(\xi \sim \mathcal{GP}(0,k)\) and set
  \[
    \hat{q}=\sqrt{1-\beta_{PCN}^{2}}\;q^{j}+\beta_{PCN}\,\xi .
  \]
  \State \textbf{Map to radial function:} \(\hat{\rho}=\exp(\hat{q})\).
  \State \textbf{Forward solve:} compute \(\hat{g}=\mathcal{G}(\hat{q})\), \(\hat{\Phi}=\Phi(\hat{q};y)\),
  \(\widehat{R}=R(\hat{q})\).
  \State \textbf{Acceptance probability:}
  \[
		\alpha = \min\left\{ 1, \frac{\exp\left(-\Phi(\hat{q};y)-\lambda R(\hat{q})\right)}{\exp\left(-\Phi(q^j;y)-\lambda R(q^j)\right)}\right\}.
		\]
  \State Draw \(u\sim \mathrm{Uniform}(0,1)\).
  \If{$u\le \alpha$} \State \(q^{j+1}\gets \hat{q}\), \(\rho^{j+1}\gets \hat{\rho}\),
         \(\Phi^{j+1}\gets \hat{\Phi}\), \(R^{j+1}\gets \widehat{R}\).
  \Else \State \(q^{j+1}\gets q^{j}\), \(\rho^{j+1}\gets \rho^{j}\),
         \(\Phi^{j+1}\gets \Phi^{j}\), \(R^{j+1}\gets R^{j}\).
  \EndIf
\EndFor
\end{algorithmic}
\end{algorithm}
\par The true center position $x_c$ is fixed at $(0, 0)$. We set the parameters $\sigma=1$, $\beta=0.01$ and $\tau=3$ for the three domains. In all numerical experiments, the PHG prior is implemented using the lifetime-dependent weight \eqref{eq:alpha-basic}. The order-dependent variant \eqref{eq:alpha-order} is not considered separately, since it reduces to the same form up to a rescaling of the regularization weight when the order factor \(\kappa_p\) is treated as a constant. Here \(\ell\) and \(p\) denote the kernel hyperparameters in \eqref{eq:seckernel}, \(d\) is the incident direction, and \(\lambda_1,\lambda_2\) are the regularization weights for the PHG and TG priors, respectively. The corresponding parameters are summarized in ~\Cref{tab:delta5} and ~\Cref{tab:delta6}.

\subsection{Comparison of different noise level}
In this subsection, we will compare the Gaussian, TV-Gaussian (TG), and PHG priors under different noise levels. Firstly, we show the reconstruction results (see \Cref{fig:kite0.05_compare}--\Cref{fig:5star0.05_compare}) for the kite, rectangle, and 5-star shaped obstacles under noise level \(\delta=5\%\), respectively. In each figure, the first three subfigures display the exact and reconstructed boundaries obtained with the Gaussian, TG, and PHG priors, while the last subfigure shows the normalized discrete \(L^2\) error of the reconstructed radial profile \(\rho=\exp(q)\) over the pCN-MCMC iterations.

Under the noise level $\delta=5\%$, all three priors are able to recover the main boundary shape to some extent. However, the Gaussian prior exhibits larger deviations from the exact boundary, especially in regions with stronger local variations. Both the TG and PHG priors provide improved reconstructions. At this noise level, the PHG prior remains broadly comparable to the TG prior, while yielding slightly smaller normalized discrete \(L^2\) errors in some examples. The improvement is more pronounced for the rectangular and 5-star shaped obstacles, whose boundaries exhibit sharper local variations and more oscillatory components. These results indicate that under this noise level, the PHG prior provides a meaningful improvement over the Gaussian prior and performs competitively with the TG prior.
\renewcommand{\arraystretch}{1.15}
\newcolumntype{C}[1]{>{\centering\arraybackslash}m{#1}}
\begin{table}[t!]
\centering
\begin{tabular}{|C{2cm}|C{2.0cm}|C{2.0cm}|C{2.5cm}|C{2.5cm}|C{2.5cm}|}
\hline
\textbf{$\delta=5\%$} & \textbf{$\ell$} & \textbf{$p$} & $d$ &\textbf{PHG weight $\lambda_1$} & \textbf{TG weight $\lambda_2$} \\
\hline
Kite      & $0.55$ & $2$  &$(1,0)$ & $220$ & $500$ \\

\hline
Rectangle & $1$    & $1$  &$(1,0) $& $15$  & $40$  \\
\hline
5-star     & $0.8$  & $0.4$ & $(1,0)$&$5.5$ & $26$  \\
\hline
\end{tabular}

\captionsetup{justification=centering,singlelinecheck=false}
\hspace{-1cm}
\caption{Reconstruction parameter for three different obstacles when $\delta=5\%$.}
\label{tab:delta5}
\end{table}

\renewcommand{\arraystretch}{1.15}
\newcolumntype{C}[1]{>{\centering\arraybackslash}m{#1}}

\begin{table}[t!]
\centering
\begin{tabular}{|C{2cm}|C{2.0cm}|C{2.0cm}|C{2cm}|C{2.5cm}|C{2.5cm}|}
\hline
\textbf{$\delta=10\%$} & \textbf{$\ell$} & \textbf{$p$} & $d$ &\textbf{PHG weight $\lambda_1$} & \textbf{TG weight $\lambda_2$}\\
\hline
Kite      & $0.8$ & $2$  & $(1,0)$ & $200$ & $550$ \\
\hline  
Rectangle & $0.3$    & $2$  &$(1,0)$ & $150$  & $510$  \\
\hline
5-star     & $0.25$  & $2$ & $(1,0)$      $(0,1)$&$1$ & $5$  \\
\hline
\end{tabular}

\captionsetup{justification=centering,singlelinecheck=false}
\caption{Reconstruction parameter for three different obstacles when $\delta=10\%$.}x
\label{tab:delta6}
\end{table}

\begin{figure}[htbp!]
    \centering
    \includegraphics[width=0.4\linewidth]{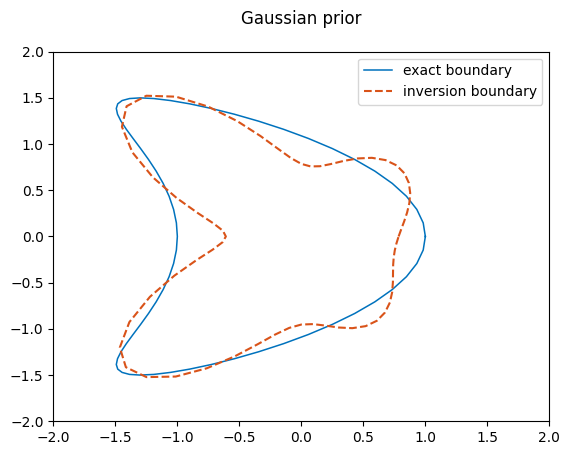}
    \includegraphics[width=0.4\linewidth]{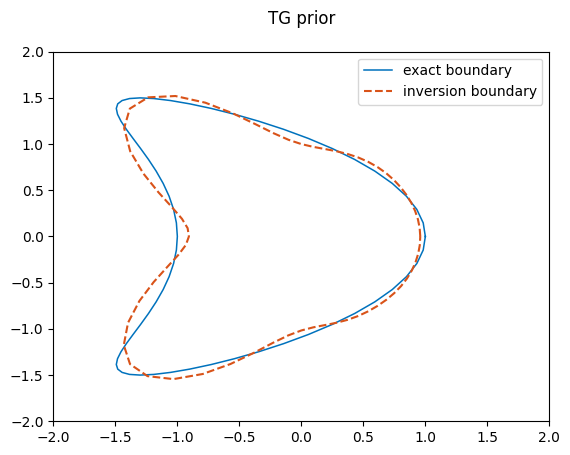}\\
    \includegraphics[width=0.4\linewidth]{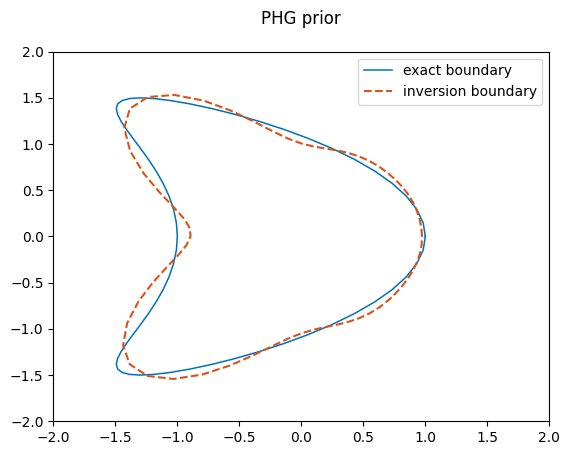}
    \includegraphics[width=0.38\linewidth]{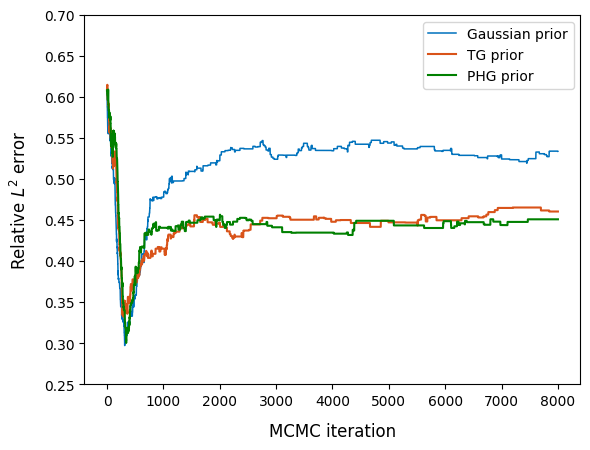}
    \caption{Reconstruction results for the kite shaped obstacle at noise level \(\delta=5\%\). The first three subfigures show the exact and reconstructed boundaries under the Gaussian, TG, and PHG priors, respectively. The last subfigure shows the relative \(L^2\) error over the pCN-MCMC iterations.}
 \label{fig:kite0.05_compare}
\end{figure}
\begin{figure}[b!]
    \centering
    \includegraphics[width=0.4\linewidth]{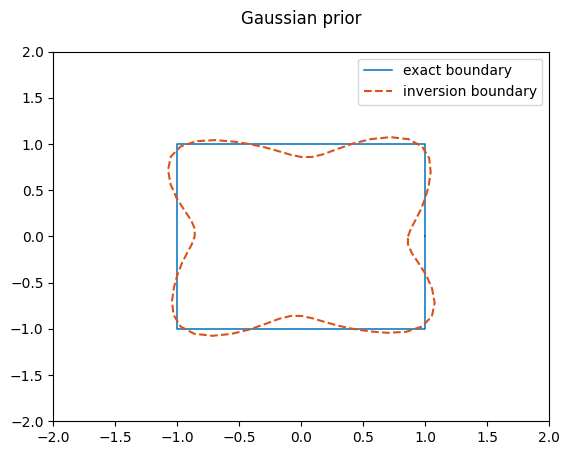}
    \includegraphics[width=0.4\linewidth]{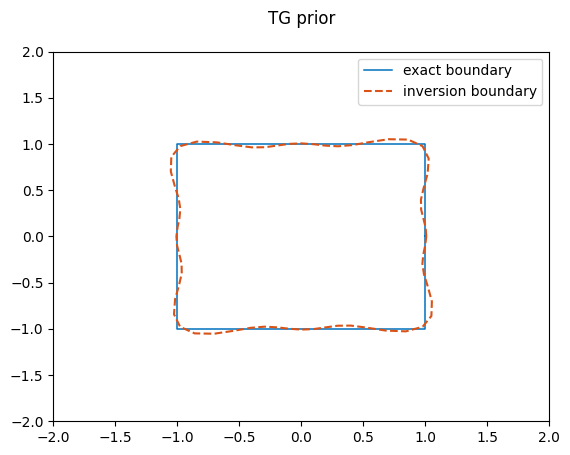}
    \includegraphics[width=0.4\linewidth]{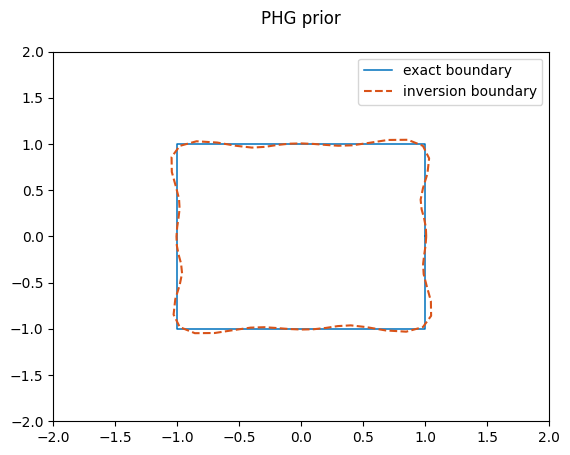}
    \includegraphics[width=0.38\linewidth]{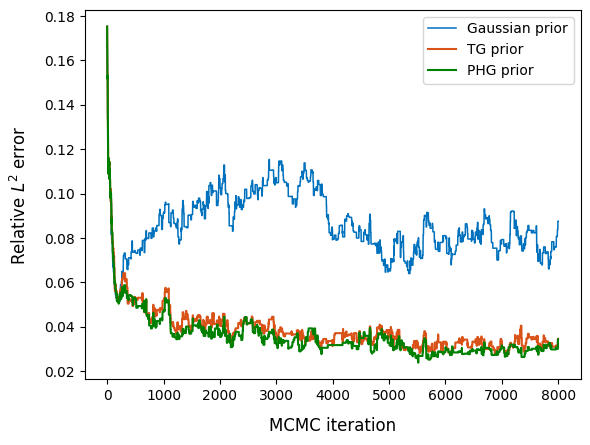}
    \caption{Reconstruction results for the rectangle shaped obstacle at noise level \(\delta=5\%\). The first three subfigures show the exact and reconstructed boundaries under the Gaussian, TG, and PHG priors, respectively. The last subfigure shows the relative \(L^2\) error over the pCN-MCMC iterations.}
 \label{fig:rectangle0.05_compare}
\end{figure}

Secondly, we also show the performance of reconstruction under the higher noise level \(\delta=10\%\), where the inverse problem becomes more challenging and the robustness of the prior plays a more prominent role. \Cref{fig:kite0.1_compare}--\Cref{fig:5star0.1_compare} show the corresponding reconstruction results for the examples of kite, rectangle, and 5-star. In each figure, the first three subfigures demonstrate the comparisons of exact and reconstructed boundaries under the Gaussian, TG, and PHG priors, while the  last subfigure shows the normalized discrete \(L^2\) error over the pCN-MCMC iterations.
\begin{figure}[t!]
    \centering
    \includegraphics[width=0.4\linewidth]{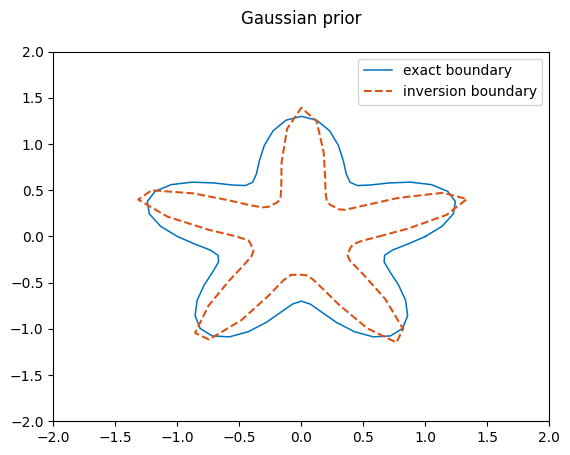}
    \includegraphics[width=0.4\linewidth]{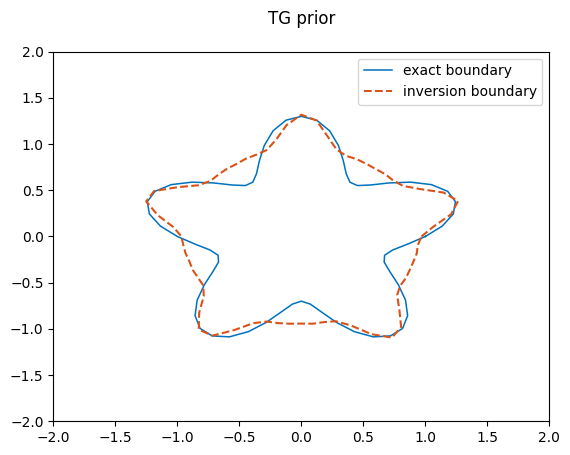}
    \includegraphics[width=0.4\linewidth]{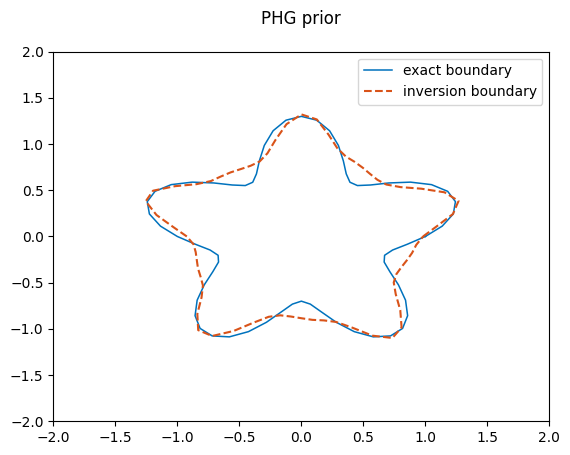}
    \includegraphics[width=0.38\linewidth]{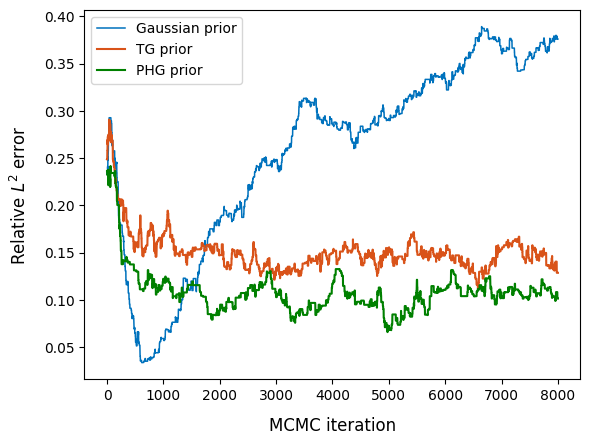}
    \caption{Reconstruction results for the 5-star shaped obstacle at noise level \(\delta=5\%\). The first three subfigures show the exact and reconstructed boundaries under the Gaussian, TG, and PHG priors, respectively. The last subfigure shows the relative \(L^2\) error over the pCN-MCMC iterations.}
 \label{fig:5star0.05_compare}
\end{figure}
\begin{figure}[b!]
    \centering
    \includegraphics[width=0.4\linewidth]{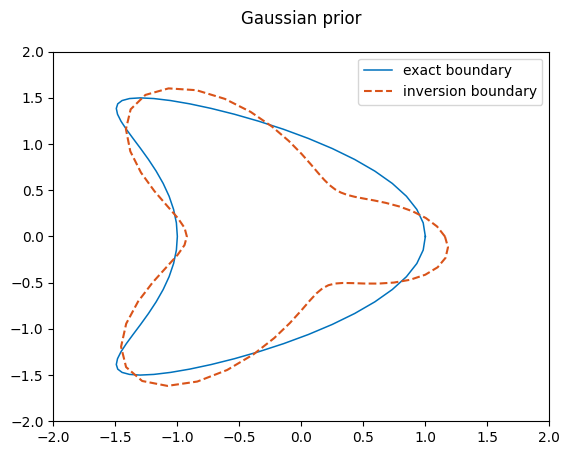}
    \includegraphics[width=0.4\linewidth]{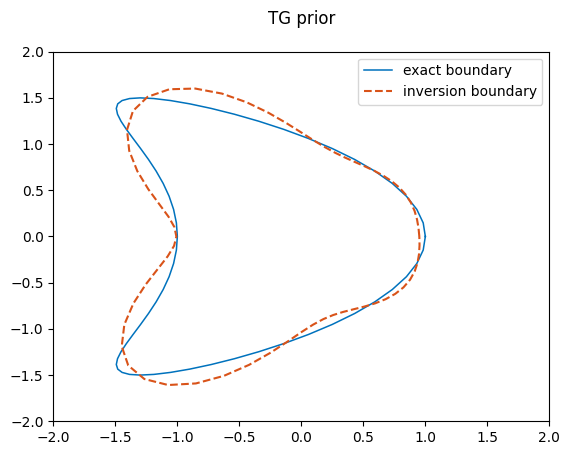}
    \includegraphics[width=0.4\linewidth]{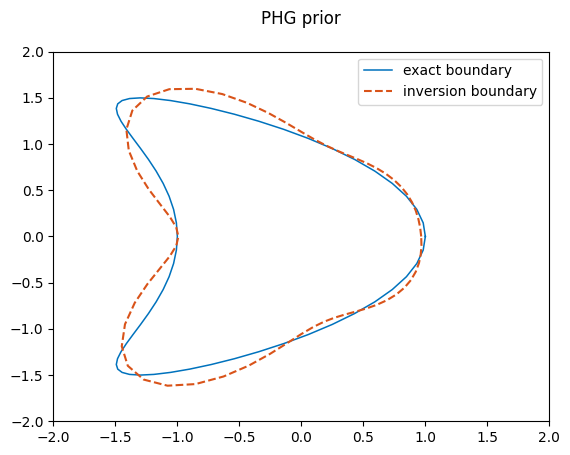}
    \includegraphics[width=0.38\linewidth]{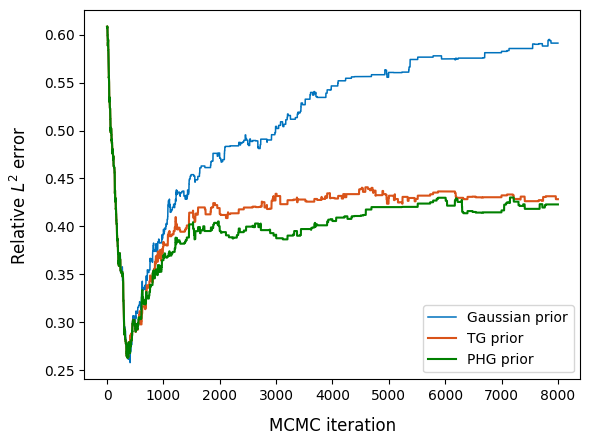}
    \caption{Reconstruction results for the kite shaped obstacle at noise level \(\delta=10\%\). The first three subfigures show the exact and reconstructed boundaries under the Gaussian, TG, and PHG priors, respectively. The last subfigure shows the relative \(L^2\) error over the pCN-MCMC iterations.}
 \label{fig:kite0.1_compare}
\end{figure}
\newpage
\begin{figure}[b!]
    \centering
    \includegraphics[width=0.4\linewidth]{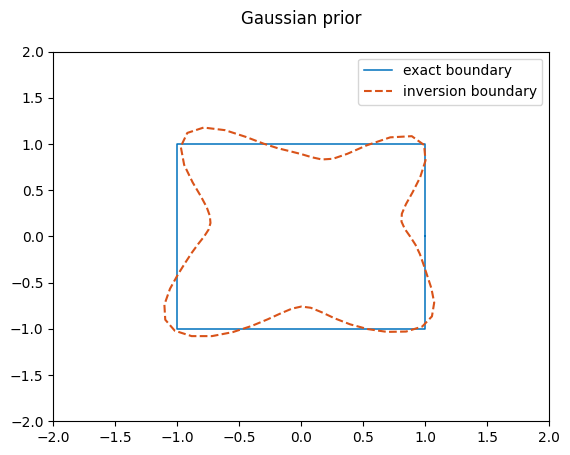}
    \includegraphics[width=0.4\linewidth]{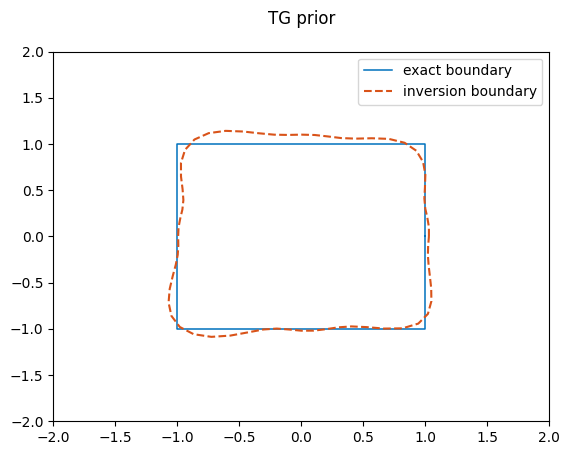}
    \includegraphics[width=0.4\linewidth]{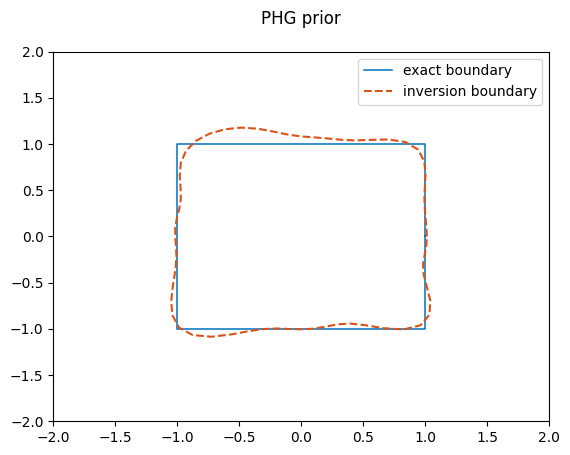}
    \includegraphics[width=0.38\linewidth]{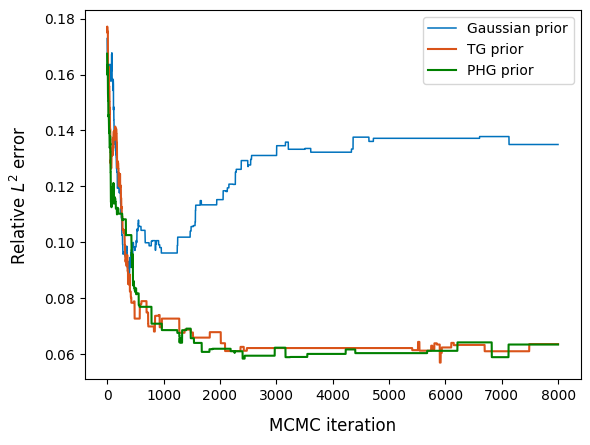}
    \caption{Reconstruction results for the rectangle shaped obstacle at noise level \(\delta=10\%\). The first three subfigures show the exact and reconstructed boundaries under the Gaussian, TG, and PHG priors, respectively. The last subfigure shows the relative \(L^2\) error over the pCN-MCMC iterations.}
 \label{fig:rectangle0.1_compare}
\end{figure}
\begin{figure}[b!]
    \centering
    \includegraphics[width=0.4\linewidth]{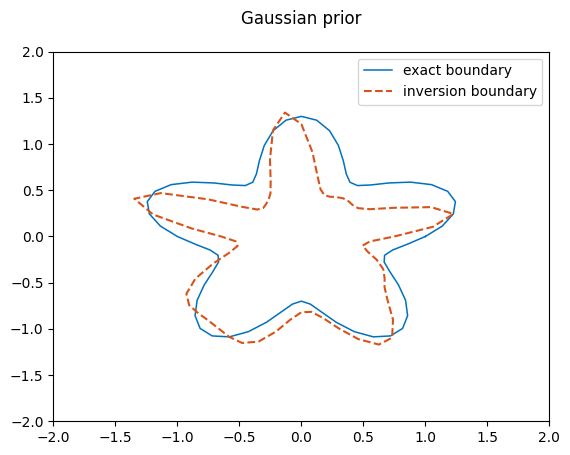}
    \includegraphics[width=0.4\linewidth]{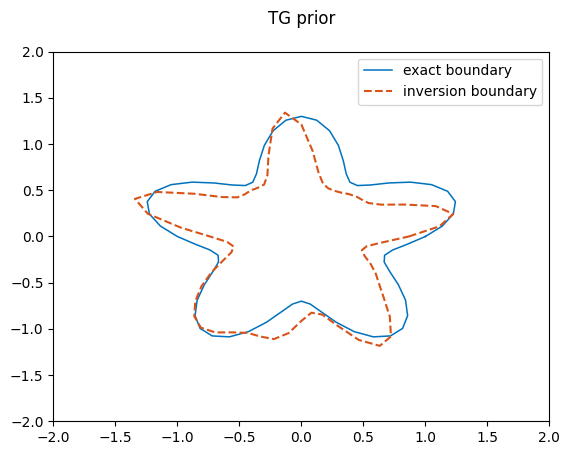}
    \includegraphics[width=0.4\linewidth]{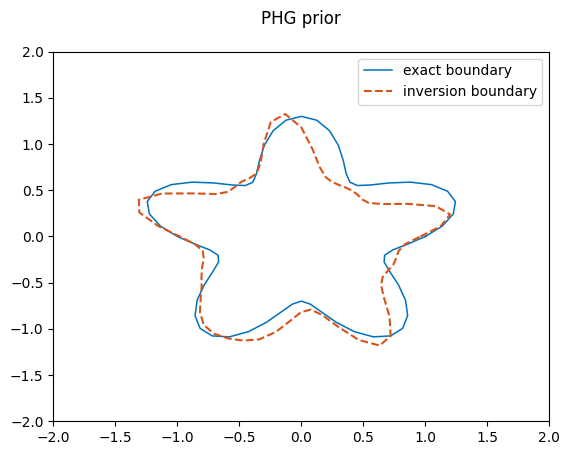}
    \includegraphics[width=0.38\linewidth]{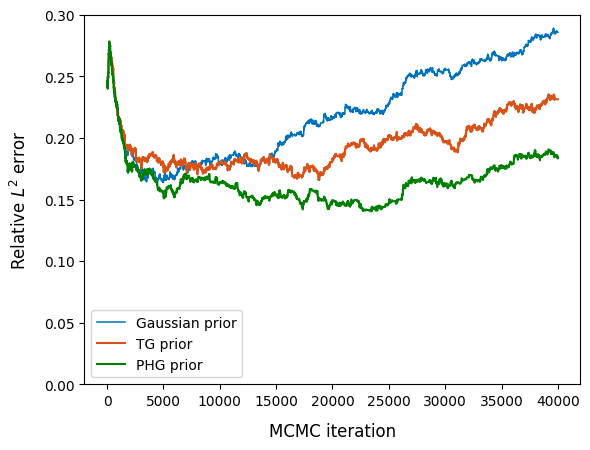}
    \caption{Reconstruction results for the 5-star shaped obstacle at noise level \(\delta=10\%\). The first three subfigures show the exact and reconstructed boundaries under the Gaussian, TG, and PHG priors, respectively. The last subfigures shows the relative \(L^2\) error over the pCN-MCMC iterations.}
 \label{fig:5star0.1_compare}
\end{figure}

With the noise level increasing to \(\delta=10\%\), the PHG prior performs more competitively than the other two priors for different obstacles. The Gaussian prior becomes less accurate at this noise level, resulting in larger deviations from the exact boundary and higher reconstruction errors. In contrast, both the TG and PHG priors remain considerably more stable. 
Furthermore, PHG prior tends to better preserve the main geometric features of the boundary, which yields lower and more stable normalized discrete \(L^2\) errors under higher noise level. Specifically, this result can be found in the 5-star case, suggesting that persistence-based regularization becomes more advantageous as both noise levels and geometric complexity increase.

\subsection{Analysis of different parameters and parameterized boundary}
The parameters in \Cref{tab:delta5} and \Cref{tab:delta6} can be interpreted together with the results in \Cref{fig:kite0.05_compare}--\Cref{fig:5star0.1_compare}. In the periodic kernel \eqref{eq:seckernel}, the parameter \(\ell\) controls the correlation length of the Gaussian reference measure, while \(p\) determines the periodic length of the kernel. The parameter \(\lambda_1\) represents the weight ratio between the persistence-based penalty and the data-misfit term. 
The different combinations of these parameters can effectively deal with  different geometries and noise levels. For the kite-shaped obstacle (see \Cref{fig:kite0.05_compare} and \Cref{fig:kite0.1_compare}), it is characterized by smooth, large-scale variations in the radial direction, and the spurious local fluctuations of reconstructed boundary can be reduced with the increasing of $\ell$. On the other hand, for the rectangle and 5-star shaped obstacles (see \Cref{fig:rectangle0.05_compare} and \Cref{fig:rectangle0.1_compare}, \Cref{fig:5star0.05_compare} and \Cref{fig:5star0.1_compare}), it exhibits sharper local variation and richer extrema. Thus, a smaller \(\ell\) allows greater local flexibility, which is more suitable for resolving sharp corners and oscillatory features. In contrast, a larger \(\ell\) may oversmooth these fine structures. The regularization parameter \(\lambda_1\) plays a critical role in noise suppression, particularly under high-noise conditions. Increasing values of \(\lambda_1\)  will reduce the effect of noise-induced transient persistence pairs in the persistence-based penalty term, thereby improving reconstruction stability. The suitable combinations of parameters (\(\ell\), \(p\), \(\lambda_1\)) in \Cref{tab:delta5} and \Cref{tab:delta6} demonstrate robust performance across varying geometric configurations and noise conditions, which achieve an optimal balance between Gaussian reference measure and persistence-based regularization. Furthermore, the \(\beta\) of the weight function \eqref{eq:alpha-basic} is fixed in the present experiments, which is consistent with a decay parameter of the denoising operation in~\cite{zheng2015application}. This decay parameter helps distinguish noise-induced short-lived persistence pairs from true structure persistence pairs. When this separation is more obvious, the proposed regularization tends to improve reconstruction stability.

\section{Conclusion}\label{se6}
This paper presents a novel persistent-homology-Gaussian (PHG) prior for infinite-dimensional Bayesian inversion in acoustic inverse scattering. The key contributions of this work are threefold: First, a topology-aware prior encodes shape-dependent features through periodic radial parameterization, which enhances computational efficiency and robustness. Second, theoretical analysis of posterior well-posedness has been derived under Hellinger, total variation, and Wasserstein-$p$ metrics, where persistence-based regularization preserves essential boundary features and suppresses noise artifacts. Third, the function space is extended to generalized topological spaces with practical finite-dimensional approximation. Compared to traditional TV-based priors operating in $BV$-type spaces, the PHG prior is gradient-free, which can be regarded as a weighted TV prior in the one-dimensional case. However, it differs significantly from the TV prior in higher dimensions, which can capture higher-order geometric complexity through persistent homology filtrations. Numerical experiments demonstrate superior performance, particularly in challenging scenarios involving high noise levels, sharp transitions, and oscillatory boundaries.

Although this study employs certain simplifying hypotheses, the proposed PHG prior and its associated computational framework demonstrate broad applicability to inverse problems involving unknowns with pronounced oscillatory behavior or shape-dependent features, particularly in scenarios of noisy measurements and discontinuous structures. Future work will focus on the extension of more effective persistence pair weighting strategies to enhance the encoding of complex topological features in higher dimensions, including multiscale phenomena and heterogeneous geometric configurations.

\section*{Acknowledgement}
The work was supported by the National Natural Science Foundation of China (No. 12271409), the Foundation of National Key Laboratory of Computational Physics, and the Fundamental Research Funds for the Central Universities.
\section*{Data availability}
Data will be made available on reasonable request.
\section*{Conflict of interest}
The authors declare that they have no conflict of interest.
\appendix

\bibliography{refs}

\end{document}